\documentclass[12pt]{amsart}
\usepackage{dsfont}
\usepackage{enumitem}

\usepackage{amsmath}
\usepackage[foot]{amsaddr}
\usepackage{amssymb}
\usepackage{amsthm}
\usepackage{epsfig}
\usepackage{url}
\usepackage{array}
\usepackage{varwidth}
\usepackage{tikz}
\usetikzlibrary{patterns}
\usepackage{multirow}
\usepackage{textcomp}
\usepackage{bbm}
\usepackage{float}
\usepackage{color}

\newtheorem{theorem}{Theorem}

\newtheorem{lemma}[theorem]{Lemma}

\newtheorem{problem}[theorem]{Problem}

\newtheorem{definition}[theorem]{Definition}

\newcommand{\F}{\mathbb{F}}
\newcommand{\Fq}{\mathbb{F}_q}
\newcommand{\Fqn}{\mathbb{F}_q^n}

\newcommand{\veclist}{\vec{x}_1,\ldots,\vec{x}_{k+1}}
\newcommand{\smallveclist}{\vec{x}_1,\ldots,\vec{x}_k}

\title{Partition Rank and Partition Lattices}

\author{Mohamed Omar}
\thanks{Department of Mathematics and Statistics, York University. 4700 Keele St. Toronto, Ontario, Canada. M3J 1P3}
\date{\today}
 
\begin{document}

\maketitle

\begin{abstract}
We introduce a universal approach for applying the partition rank method, an extension of Tao's slice rank polynomial method, to tensors that are not diagonal. This is accomplished by generalizing Naslund's distinctness indicator to what we call a \emph{partition indicator}. The advantages of partition indicators are two-fold: they diagonalize tensors that are constant when specified sets of variables are equal, and even in more general settings they can often substantially reduce the partition rank as compared to when a distinctness indicator is applied. The key to our discoveries is integrating  the partition rank method with M\"{o}bius inversion on the lattice of partitions of a finite set. Through this we unify disparate applications of the partition rank method in the literature. We then use our theory to address a finite field analogue of a question of Erd\H{o}s, thereby generalizing results of Hart and Iosevich and independently Shparlinski. Furthermore we generalize work of Pach, et al. on bounding sizes of sets avoiding right triangles to bounding sizes of sets avoiding right $k$-configurations.
\end{abstract}





\section{Introduction}

A fundamental concern in extremal combinatorics is how large a subset of an ambient set can be while avoiding a certain property. A hallmark example of such a problem that has received particular attention over the past few decades is the Cap Set problem, which asks to determine the largest size of a subset of $\mathbb{F}_3^n$ avoiding a triple of collinear points. One can see this is precisely the question of determining the largest size of a subset of $\mathbb{F}_3^n$ that avoids a $3$-term arithmetic progression. This question was raised independently in differing contexts \cite{alon1995lattice}, \cite{calderbank1994maximal}, \cite{frankl1987subsets}, dating back as early as the 1980s. In 1995, Meshulam \cite{meshulam1995subsets} proved an upper bound of $O(3^n/n)$ and not until a 2012 article by Bateman and Katz \cite{bateman2012new}  was this significantly tightened to an upper bound of $O(3^n/n^{1+\epsilon})$ for a positive constant $\epsilon$. Despite these bounds, it was conjectured and was a major open problem in additive combinatorics that one could achieve an exponentially small upper bound of $c^n$ with $c<3$. This was finally resolved in a 2017 article by Ellenberg and Gijswijt \cite{capset} building on a breakthrough of Croot, Lev and Pach \cite{crootlevpach} for the analogous $3$-term arithmetic progression problem on $(\mathbb{Z}/4\mathbb{Z})^n$. The key insight was using what Tao \cite{tao2016notes1} later coined as the \emph{slice rank polynomial method}.

The techniques of \cite{crootlevpach} and \cite{capset}  in the language of \cite{tao2016notes1} paved the way for breakthroughs on a plethora of disparate extremal combinatorics problems. Tao \cite{tao2016notes1} used it to find exponentially small lower bounds on sizes of subsets of $\mathbb{F}_p^n$ (with $p$ prime) that guarantee solutions to a homogeneous linear system with variables not all equal. Naslund and Sawin \cite{naslund2017upper} used it to make breakthroughs on the Sunflower Conjecture of Erd\H{o}s and Szemer\'{e}di \cite{erdos1978combinatorial}. Blasiak, et al. used slice rank to find exponentially small bounds on tricolored sum-free sets in abelian groups \cite{blasiak2016cap}.

Despite the ubiquitous applications of the method, it has limitations in providing upper bounds on sizes of sets that avoid $k$-tuples of elements satisfying specific properties when $k \geq 4$. To address this, Naslund \cite{naslund2020partition} introduced the \emph{partition rank method}, and used it to extend the arguments of Ge and Shangguan \cite{ge2020maximum} and obtain a polynomial bound on the maximum size of subsets of $\mathbb{F}_q^n$ (the polynomial in $n$ with exponent in $q$) containing no $k$-right corner, that is $k$ right angles emanating from a point. Naslund further used this approach to improve the best known Erd\H{o}s-Ginzburg-Ziv constants for finite abelian groups \cite{naslund2020exponential} and more recently to address problems in Euclidean Ramsey Theory \cite{naslund2022chromatic}.

Many of the applications of slice rank we've discussed use diagonal tensors to establish bounds on sizes of sets. However, this is not always the case. The work of Ge and Shangguan \cite{ge2020maximum} uses the slice rank method together with a non-diagonal $3$-tensor to find polynomial upper bounds on sizes of subsets in $\Fqn$ ($q$ a power of an odd prime) that avoid a right angle, drastically improving the previous exponential bound due to Bennett \cite{bennett2018occurrence}. The work of Sauermann \cite{sauermann2022finding} uses a non-diagonal tensor together with a characterization of slice rank by Sawin and Tao \cite{tao2016notes} to find exponentially small lower bounds on sizes of sets in $\F_p^n$ ($p$ an odd prime) that guarantee the existence of solutions to certain linear systems where variables are all distinct. However, non-diagonal tensors become more difficult to work with when applying the partition rank method. Naslund mitigates this by introducing the concept of a \emph{distinctness indicator} that diagonalizes a non-diagonal tensor after multiplication by it \cite{naslund2020partition}.

This article introduces a more universal approach for applying the partition rank method in the situation that the tensors in play are non-diagonal by generalizing Naslund's distinctness indicator to what we call a \emph{partition indicator}. Roughly speaking, this allows one to diagonalize a tensor if its behavior is controlled when specified collections of variables are set equal. The key insight is merging the partition rank construction with the theory of partially ordered sets, particularly as it pertains to the lattice of partitions of a set. We not only recover theory in the literature through this framework, but we see how to apply the framework to discover new results. Using our first main theorem, Theorem~\ref{thmmain1}, we find a  polynomial upper bound on sizes of subsets of $\Fqn$ ($q$ an odd prime power) only forming acute angles among them, addressing a finite field analogue of a problem of Erd\H{o}s \cite{erdos} and generalizing work of Hart and Iosevich \cite{hart2008ubiquity} and Shparlinski \cite{shparlinski2010point}. Using our second main theorem, 
we find polynomial upper bounds on sizes of subsets of $\Fqn$ avoiding a general class of properties, an instance of which extends a result of Pach, et al. \cite{pach2020avoiding} from bounding the maximum size of subsets of $\Fqn$ containing no right triangle to containing no right $k$-configuration (that is, $k$ distinct points, every triple of which form right angles). We also re-derive work of Naslund used to improve bounds on Erd\H{o}s-Ginzburg-Ziv constants. 

Our article is organized as follows. We begin in  Section~\ref{secprelim} with preliminaries that introduce set partitions, partition lattices, and the slice rank and partition rank methods. In Section~\ref{secpartitions} we develop our main tool, the partition indicator. We then integrate it into the theory of partition rank to provide opportunity for more widespread applications.  Section~\ref{secapplications} is dedicated to applications of the theory from Section~\ref{secpartitions}, proving the new results bounding sizes of sets having only acute angles among them, right $k$-configurations, and more. We finish in Section~\ref{secfuturedirections} with future directions. 




\section{Preliminaries}\label{secprelim}

\subsection{Partitions}
We start by introducing necessary theory on set partitions. Much of this can be found in \cite{stanley2011enumerative}. Let $k$ be a positive integer. Denote by $\Pi_k$ the set of partitions of $\{1,2,\ldots,k\}$. For any $\pi \in \Pi_k$, we write $\pi$ as $\pi_1\pi_2 \cdots \pi_r$ if $\pi$ is partitioned into $r$ sets (each of these sets is called a block), where $\pi_i$ lists the contents in the $i$th set of the partition for each $i$ (with order not mattering). For instance, if $k=8$ and $\pi=\{\{1,2,6,8\},\{3,4\},\{5,7\}\}$ then $\pi=\pi_1\pi_2\pi_3$ with $\pi_1=1268, \ \pi_2=34, \ \pi_3=57$. Sometimes we write $\pi$ as a contiguous list of integers separated by a vertical bars to denote the separation of blocks. For instance the previous partition may be written as $1268 \vert 34 \vert 57$. An explicit expression for the number of partitions in $\Pi_k$ with exactly $r$ blocks is typically complicated, and so is given its own notation $S(k,r)$. The numbers $S(k,r)$ as $k$ and $r$ vary are known as the \emph{Stirling numbers of the second kind}. Explicit formulas for the Stirling numbers of the second kind are generally  cumbersome, but there are many identities they satisfy (see \cite{stanley2011enumerative}). 

The set $\Pi_k$ forms a partial order $<$ under refinement, meaning if $\pi'$ and $\pi$ are distinct  then $\pi' < \pi$ if every block of $\pi'$ lies in a block of $\pi$. For instance if $\pi'=16 \vert 28 \vert 34 \vert 5 \vert 7$ then $\pi'<\pi$ where $\pi$ is the aforementioned partition above. With this relation, we say $\pi'$ is finer than $\pi$ and $\pi$ is coarser than $\pi'$.  In fact, $\Pi_k$ under the partial order $<$ forms a lattice: every pair of elements has a unique least upper bound and unique greatest lower bound. The lattice has a unique minimal element $\hat{0}$ which consists of the $k$ blocks each of size one and a unique maximal element $\hat{1}$ which consists of the one block $12 \cdots k$. The partitions $\hat{0}$ and $\hat{1}$ being minimal and maximal respectively means $\hat{0} \leq \pi \leq \hat{1}$ for every $\pi \in \Pi_k$. This lattice is also graded: for any $\pi \in \Pi_k$, the \emph{rank} of $\pi$, denoted $\mbox{rk}(\pi)$, is $k$ minus the number of blocks in $\pi$. In this way, $\mbox{rk}(\hat{0})=0, \ \mbox{rk}(\hat{1})=k-1$ and for any $\pi,\pi' \in \Pi_k$ with $\pi$ covering $\pi'$ (meaning no partitions lie between these two partitions under $<$), we have $\mbox{rk}(\pi)=\mbox{rk}(\pi')+1$.

For any finite set $P$ with a partial order $<$ on it, and any field $\mathbb{F}$, one can define the \emph{zeta function} of $P$ on pairs of elements of $P$ given by $\zeta: P \times P \to \mathbb{F}$ with $\zeta(x,y)=1$ if $x \leq y$ and $0$ otherwise. Viewing $\zeta$ as a matrix in $\mathbb{F}^{P \times P}$, with the row and column indices ordered by a linear extension of $P$, $\zeta$ is upper triangular with $1$'s along the diagonals. It therefore has a matrix inverse, which is typically denoted by $\mu$ and is called the \emph{M\"{o}bius function} of $P$ with respect to $<$. Fundamental relations satisfied by $\mu$ are:
\begin{equation}\label{eqnmobius}
\mu(x,y) = \begin{cases}
1  \hspace{0.25in} &\mbox{ if } x=y \\
-\sum_{x \leq z < y} \mu(x,z)  \hspace{0.25in} &\mbox{ if } x<y \\
0  \hspace{0.25in} &\mbox{ otherwise}
\end{cases}
\end{equation}
Observe that the second condition is equivalent to stating that for $x<y$, $\sum_{x \leq z \leq y} \mu(x,z)=0$, a reformulation that is fruitful for our endeavors. The M\"{o}bius function of the poset $\Pi_k$ under refinement is completely understood. Firstly, one can show that $\mu(\hat{0},\hat{1})=(-1)^{k-1}(k-1)!$. Now suppose that $\pi$ has $r$ blocks, $\pi'<\pi$, and that the $i$th block of $\pi$ is split into $\lambda_i$ blocks in $\pi'$. Then the interval $[\pi',\pi]$ in $\Pi_k$ is isomorphic to the product poset $\Pi_{\lambda_1} \times \cdots \times \Pi_{\lambda_r}$ and one can use this to show that $\mu(\pi',\pi)=(-1)^{\lambda_1-1} (\lambda_1-1)! \cdots (-1)^{\lambda_r-1} (\lambda_r-1)!$.

In many applications of the partition rank method we look at functions defined on $k$-tuples from a set where certain members of the $k$-tuples are deemed to be equal. In that light, given a set $X$, field $\mathbb{F}$, and $\pi \in \Pi_k$, define $\delta_{\pi}:X^k \to \mathbb{F}$ by
\[
\delta_{\pi}(x_1,\ldots,x_k) = \begin{cases}
1  \hspace{0.25in} &\mbox{ if } x_i=x_j \mbox{ for every } i,j \mbox{ in the same block of } \pi \\
0  \hspace{0.25in} &\mbox{ otherwise}
\end{cases}
\]
For instance if $k=4$ and $x_1=x_4$ are different than $x_2=x_3$, then $\delta_{\pi}(x_1,x_2,x_3,x_4)=1$ if and only if $\pi \leq 14 \vert 23$. Notice then that $14 \vert 23$ is the coarsest partition $\pi$ for which $\delta_{\pi}(x_1,x_2,x_3,x_4)=1$, and is truly the partition that identifies which variables are equal and which are not. To capture this, for a $k$-tuple $(x_1,x_2,\ldots,x_k) \in X^k$ we define the \emph{partition} of $(x_1,x_2,\ldots,x_k)$ to be the coarsest partition $\pi$ for which $\delta_{\pi}(x_1,\ldots,x_k)=1$. Of particular note is that the partition of $(x_1,x_2,\ldots,x_k)$ is $\hat{0}$ if and only if the $x_i$'s are distinct and is $\hat{1}$ if and only if the $x_i$'s are all equal.

\subsection{Slice and Partition Rank}

We start by introducing the concept of the slice rank of a tensor.

\begin{definition}
Let $X_1,\ldots,X_k$ be finite sets and $T:X_1 \times \cdots \times X_k \to \F$. We say $T$ is a \emph{slice} if it can be written as
\[
T(x_1,\ldots,x_k) = T'(x_i) \cdot T''(x_1,\ldots,x_{i-1},x_{i+1},\ldots,x_k)
\]
for some $i \in \{1,2,\ldots,k\}$, where $T':X_i \to \F$ and $T'':X_1 \times \cdots X_{i-1} \times X_{i+1} \times \cdots \times X_k \to \F$. The \emph{slice rank} of $T$, written \mbox{slice-rank}(T), is the minimum value $r$ such that $T$ can be written as a sum of $r$ slices.
\end{definition}

For instance, the tensor $T:\F_3^3 \to \F_3$ given by $T(x,y,z)=xy+xz+yz$ has slice rank $2$ because we can not factor out a function of any one variable, but $T(x,y,z)=T_1'(x) \cdot T_1''(y,z)+T_2'(y) \cdot T_2''(x,z)$ where
\[
T_1'(x)=x, \ T_1''(y,z)=y+z, \ T_2'(y)=y, \ T_2''(x,z)=z.
\]
Suppose now that $X$ is a set and $A \subseteq X$ with the caveat that $k$-tuples of $A$ avoid a particular property $\mathcal{P}$. In many applications, authors construct a tensor $T:A^k \to \mathbb{F}$ for some field $\mathbb{F}$, and exploit the fact that $A$ avoids $\mathcal{P}$ to make $T$ a diagonal tensor with non-zero diagonal entries. Once this is done, an upper bound on $ \vert A \vert $ can be established through the slice rank of $T$. 

\begin{theorem}[\cite{tao2016notes}]\label{thmslicerank}
Let $A$ be a finite set and $\F$ be a field. Let $T:A^k \to \F$ be such that $T(x_1,x_2,\ldots,x_k) \neq 0$ if and only if $x_1=x_2=\cdots=x_k$. Then the slice rank of $T$ is $ \vert A \vert $.
\end{theorem}

Theorem~\ref{thmslicerank} is particularly useful when $T$ is a polynomial: one can estimate the slice rank by peeling off monomials of low degree in each summand, and counting the number of possible ways to do so. This was the crux of the key arguments in the seminal papers of Croot, Lev, Pach \cite{crootlevpach} and Ellenberg, Gijswijt \cite{capset}. 

With the concept of slice rank established, we now introduce the concept of partition rank. First we need some terminology. Given a nonempty proper subset $S=\{i_1,\ldots,i_m\}$ of $\{1,2,\ldots,k\}$ and a tensor $T: X_{i_1} \times \cdots \times X_{i_m} \to \F$ we abbreviate the expression $T(x_{i_1},\ldots,x_{i_m})$ by $T(\vec{x}_S)$. We interchangeably write $T(\vec{x}_{\pi_i})$ when $S$ happens to be the contents of a block $\pi_i$ of a partition $\pi$. 

\begin{definition}\label{defpartition}
Let $X_1,\ldots,X_k$ be sets and $T:X_1 \times \cdots \times X_k \to \mathbb{F}$. We say $T$ has \emph{partition rank $1$} if there is a nontrivial partition $\pi=\pi_1 \cdots \pi_r$ in $\Pi_k$ and tensors $T_1,T_2,\ldots,T_r$ with $T_i: \prod_{j \in \pi_i} X_j \to \mathbb{F}$ so that
\[
T(x_1,\ldots,x_k) = \prod_{i=1}^r T_i(\vec{x}_{\pi_i}).
\]
The \emph{partition rank} of $T$, written \mbox{partition-rank}(T), is the minimum value $r$ such that $T$ can be written as the sum of $r$ functions of partition rank $1$.
\end{definition}
In other words, a tensor $T$ has partition rank $1$ if it can be written as a product of tensors, each factor using distinct sets of the variables $x_1,\ldots,x_k$. As an example, $T:X^4 \to \F$ given by $T(x_1,x_2,x_3,x_4)=(x_1^2+x_2^2) \cdot \delta(x_3,x_4)$ has partition rank $1$, but its not immediate how to calculate its slice rank for a generic $\F$.

A significant reason to introduce partition rank is that there is an analogue of Theorem~\ref{thmslicerank} that allows us to bound the size of a set using it.

\begin{theorem}[\cite{naslund2020partition}]\label{thmpartitionrank}
Let $A$ be a finite set and $\F$ be a field. Let $T:A^k \to \F$ be a function such that $T(x_1,x_2,\ldots,x_k) \neq 0$ if and only if  $x_1=x_2=\cdots=x_k$. Then the partition rank of $T$ is $ \vert A \vert $.
\end{theorem}

At first glance we notice that if $T:A^k \to \mathbb{F}$ is a diagonal tensor with non-zero diagonal entries then Theorem~\ref{thmslicerank} and Theorem~\ref{thmpartitionrank} imply $ \vert A \vert =\mbox{slice-rank}(T)=\mbox{partition-rank}(T)$ so the partition rank may at first seem to not provided added value. However, in practice, an upper bound on $ \vert A \vert $ is established through these theorems by finding an upper bound on $\mbox{slice-rank}(T)$ in Theorem~\ref{thmslicerank} and analogously $\mbox{partition-rank}(T)$ in Theorem~\ref{thmpartitionrank}. In general we see that the partition rank offers more flexibility than the slice rank: if we have a slice $T:X_1 \times \cdots \times X_k \to \F$, then it can be written as $T'(x_i) \cdot T''(x_1,\ldots,x_{i-1},x_{i+1},\ldots,x_k)$, which has partition rank $1$, but partition rank allows us to split our tensors in many more ways than just peeling off a tensor of one variable. Because of this, it is potentially much easier to get tighter upper bounds on partition-rank($T$) as opposed to slice-rank($T$). 




\section{Partition Rank \& Partition Lattices}\label{secpartitions}

This section presents the main theory of the article, introducing partition indicators and how they provide more universal opportunities to use the partition rank method. We start with a recap of how partition rank is concretely used in literature to develop upper bounds on set sizes, as motivation for the theory we build.

\subsection{Distinctness Indicators}

In \cite{naslund2020partition}, Naslund proved that if $A \subseteq \Fqn$ (here $q$ an odd prime power) contains no \emph{$k$-right-corner}, that is no distinct vectors $\vec{x}_1,\ldots,\vec{x}_{k+1}$ for which $\vec{x}_1-\vec{x}_{k+1},\ldots,\vec{x}_k-\vec{x}_{k+1}$ are mutually orthogonal, then $ \vert A \vert  \leq \binom{n+(k-1)q}{(k-1)(q-1)}$. To find such an upper bound one might attempt to create a diagonal tensor $T:A^{k+1} \to \Fq$, find an upper bound for  $\mbox{slice-rank}(T)$, and apply Theorem~\ref{thmslicerank}. A reasonable start would be the tensor
\[
T(\veclist) = \prod_{j<\ell \leq k} \left(1-\langle \vec{x}_j-\vec{x}_{k+1},\vec{x}_{\ell}-\vec{x}_{k+1} \rangle^{q-1} \right).
\]
We see that 
\[
T(\veclist)=\begin{cases}
1  \hspace{0.15in} &\mbox{ if } \vec{x}_1-\vec{x}_{k+1},\ldots,\vec{x}_k-\vec{x}_{k+1}  \mbox{ are mutually orthogonal } \\
0  \hspace{0.15in} &\mbox{ otherwise }
\end{cases}
\]
so one might suspect then that $T$ is diagonal since $A$ avoids $k$-right corners. However this is subtly not the case. A $k$-right corner requires vectors to be distinct, and it is possible for $T(\veclist)=1$ even if that is not the case. For instance, we can set $\vec{x}_2=\cdots=\vec{x}_k$, have $\vec{x}_1-\vec{x}_{k+1}$ and $\vec{x}_2-\vec{x}_{k+1}$ orthogonal, and meanwhile have $\vec{x}_2-\vec{x}_{k+1}$ self-orthogonal (a possibility in $\Fqn$). To remedy this, Naslund augments $T$ by introducing what he coined as the \emph{distinctness indicator} function $H:A^{k+1} \to \F$, which can be written as a sum of the form
\begin{equation}\label{eqnindicator}
H(\vec{x}_1,\ldots,\vec{x}_{k+1}) = \sum_{\pi \in \Pi_{k+1}} f(\pi) \cdot \delta_{\pi}(\veclist)
\end{equation}
for strategically chosen constants $f(\pi)$ so as to have
\begin{align*}
H(\veclist) = \begin{cases}
1  \hspace{0.25in} &\mbox{ if } \veclist \mbox{ are distinct } \\
(-1)^k \cdot k!  \hspace{0.25in} &\mbox{ if } \vec{x}_1=\cdots=\vec{x}_{k+1} \\
0  \hspace{0.25in} &\mbox{ otherwise.}
\end{cases}
\end{align*}
As a concrete illustration, it turns out when $k=2$ that
\begin{align*}
H(\vec{x}_1,\vec{x}_2,\vec{x}_3) &= 1 \cdot \delta(\vec{x}_1)\delta(\vec{x}_2)\delta(\vec{x}_3) -1 \cdot \delta(\vec{x}_3)\delta(\vec{x}_1,\vec{x}_2) -1 \cdot \delta(\vec{x}_2)\delta(\vec{x}_1,\vec{x}_3) \\
&-1 \cdot \delta(\vec{x}_1)\delta(\vec{x}_2,\vec{x}_3)+0 \cdot \delta(\vec{x}_1,\vec{x}_2,\vec{x}_3).
\end{align*}
This now makes $H(\veclist) \cdot T(\veclist)$ diagonal on $A^k$ with non-zero diagonal entries. More relevant to us is that there is a specific strategy for computing the partition rank of $H \cdot T$, and hence by Theorem~\ref{thmpartitionrank}, a strategy for finding bounds on sizes of restricted sets. The strategy is afforded by the fact that each summand $\delta_{\pi}(\veclist) \cdot T(\veclist)$ in $H \cdot T$ will separate variables when $\pi \neq \hat{1}$, and have controllable partition rank.

This approach to the $k$-right corners problem exemplifies the typical and natural approach used when applying the partition rank method to find an upper bound on the size of a set $A$ for which $k$-tuples avoid a property $\mathcal{P}$.

\begin{enumerate}
    \item[(a)] Find a tensor $T:A^k \to \mathbb{F}$ for some field $\mathbb{F}$ so that
    \[
T(x_1,\ldots,x_k)=\begin{cases}
c_1  \hspace{0.25in} &\mbox{ if } (x_1,\ldots,x_k) \mbox{ satisfies } \mathcal{P} \\
c_2  \hspace{0.25in} &\mbox{ if } x_1=\cdots=x_k \\
0  \hspace{0.25in} &\mbox{ otherwise. }
\end{cases}
\]
for some nonzero $c_2$.
\item[(b)] Introduce a tensor $H:A^k \to \mathbb{F}$ akin to that in Equation~$(\ref{eqnindicator})$ so that
\[
H(x_1,\ldots,x_k) = \begin{cases}
1  \hspace{0.25in} &\mbox{ if } x_1,\ldots,x_k \mbox{ are distinct } \\
c_3  \hspace{0.25in} &\mbox{ if } x_1=\cdots=x_k\\
0  \hspace{0.25in} &\mbox{ otherwise}
\end{cases}
\]
for some nonzero $c_3$.
\item[(c)] $H \cdot T$ is then diagonal with non-zero diagonal entries on $A^k$, so $ \vert A \vert  \leq \mbox{partition-rank}(H \cdot T)$.
\end{enumerate}
The typical conundrum that arises is that it is not always simple to find a tensor  $T:A^k \to \F$ that behaves as in (a). For instance, the tensor introduced in \cite{ge2020maximum} to bound sizes of sets in $\Fqn$ with no right angles is $T:A^3 \to \mathbb{F}_q$ given by
\begin{equation}\label{eqnrightangle}
T(\vec{x},\vec{y},\vec{z})=\delta_{y \neq z} \cdot \langle \vec{y}-\vec{x},\vec{z}-\vec{x} \rangle^{q-1}.
\end{equation}
This is $1$ if $\vec{x},\vec{y},\vec{z}$ are distinct (because $A$ has no triple of vectors forming a right angle) and is $0$ if any pair of vectors are equal. Despite $T$ being non-diagonal, $T$ is constant on $3$-tuples $(\vec{x},\vec{y},\vec{z}) \in A^3$ that have the same partition. One of our main theorems, Theorem~\ref{thmmain1}, explicitly finds an upper bound for $ \vert A \vert $ in terms of $\mbox{partition-rank}(T)$ in such cases. The other, Theorem~\ref{thmmain2}, is more subtle. It replaces the distinct indicator $H$ with our partition indicator, suited for the particular tensor at hand. Multiplying a non-diagonal tensor by an appropriate partition indicator works especially well in the case when the tensor is controlled on $k$-tuples with the same partition. 

\subsection{Partition Indicators}

Let $f:\Pi_k \to \mathbb{F}$ and $A$ be a set of interest. We construct a tensor $I_f:A^k \to \mathbb{F}$ so that, unless $x_1=\cdots=x_k$, $I_f(x_1,\ldots,x_k) = f(\pi)$ where $\pi$ is the partition of $(x_1,\ldots,x_k)$. In this way, if $T:A^k \to \mathbb{F}$, we can aim to use $I_f$ to diagonalize $T$ by replacing it by $I_f \cdot T$. For instance, suppose $T:A^k \to \mathbb{F}$ is $1$ when $(x_1,\ldots,x_k)$ satisfies a property $\mathcal{P}$ and $0$ otherwise, and furthermore that $k$-tuples from $A$ avoid property $\mathcal{P}$. It is then not necessarily the case that $T$ is a diagonal tensor. Now set $f(\hat{0})=1$ and $f(\pi)=0$ for $\pi<\hat{1}$. Then $I_f \cdot T$ will be $0$ on off-diagonal elements of $A^k$ because $T$ is $0$ on distinct $k$-tuples in $A^k$, and $f(\pi)=0$ when $\pi<\hat{1}$. So, $I_f \cdot T$ is $0$ on all $k$-tuples except possibly those with $x_1=\cdots=x_k$.

\begin{definition}[Partition Indicator]
Let $f:\Pi_k \to \mathbb{F}$ for some field $\mathbb{F}$. For a set $A$ define the \emph{partition indicator} of $f$ to be the tensor $I_f:A^k \to \mathbb{F}$ given by  
\[
I_f(x_1,\ldots,x_k) = 
\sum_{\substack{\tau \in \Pi_k \\ \tau < \hat{1}}} \left( \sum_{\pi \in \Pi_k} f(\pi)  \cdot \mu(\pi,\tau) \right) \cdot \delta_{\tau}(x_1,\ldots,x_k).
\]
\end{definition}
The following result is similar in spirit to Lemma 15 in \cite{naslund2020partition}.

\begin{lemma}\label{lempartitionindicator}
\[\displaystyle
I_f(x_1,\ldots,x_k) = \begin{cases}
-\sum_{\pi < \hat{1}} f(\pi) \cdot \mu(\pi,\hat{1})  \hspace{0.25in} &\mbox{ if } x_1=x_2=\cdots=x_k \\
f(\pi)  \hspace{0.25in} & \mbox{ if } (x_1,\ldots,x_k) \mbox{ has partition } \pi \neq \hat{1}
\end{cases}
\]
\end{lemma}

\begin{proof}
Observe that if $x_1=x_2=\cdots=x_k$ then we have that
\begin{align*}
I_f(x_1,\ldots,x_k) =\sum_{\substack{\tau \in \Pi_k \\ \tau < \hat{1}}} \left( \sum_{\pi \in \Pi_k} f(\pi)  \cdot \mu(\pi,\tau) \right)  &= \sum_{\tau<\hat{1}} \sum_{\pi \leq \tau} f(\pi) \cdot \mu(\pi,\tau) \\
&= \sum_{\pi<\hat{1}} f(\pi) \cdot \left( \sum_{\pi \leq \tau < \hat{1}} \mu(\pi,\tau)
\right) \\
&=-\sum_{\pi < \hat{1}} f(\pi) \cdot \mu(\pi,\hat{1}),
\end{align*}
the last equality coming from Equation~(\ref{eqnmobius}). Otherwise, $(x_1,\ldots,x_k)$ has partition $\pi<\hat{1}$ so $\delta_{\tau}(x_1,\ldots,x_k)=0$ unless $\tau \leq \pi$ in which case it is $1$. From this,
\[
I_f(x_1,\ldots,x_k) = \sum_{\tau \leq \pi} \sum_{\rho  \leq \tau } f(\rho)  \cdot \mu(\rho,\tau) = \sum_{\rho \leq \pi} f(\rho) \cdot \left( \sum_{\rho \leq \tau \leq \pi} \mu(\rho,\tau) \right).
\]
Again by Equation~(\ref{eqnmobius}) we have $\sum_{\rho \leq \tau \leq \pi} \mu(\rho,\tau) = 0$ unless $\rho=\pi$ in which case it is $\mu(\pi,\pi)=1$, so the sum is indeed $f(\pi)$.
\end{proof}
We can reconcile this with the distinctness indicator. Set $f(\pi)=0$ for all $\hat{0} < \pi \leq \hat{1}$ and $f(\hat{0})=1$. Then by Lemma~\ref{lempartitionindicator}, $I_f$ is $1$ when all $x_i$'s are distinct, $0$ for all other $k$-tuples $(x_1,\ldots,x_k)$ that are not all equal, and if $x_1=\cdots=x_k$, $I_f(x_1,\ldots,x_k)=-\mu(\hat{0},\hat{1})$. Since $\mu(\hat{0},\hat{1})=(-1)^{k-1}(k-1)!$, $I_f$ agrees with the distinctness indicator up to sign.

We now state our first main theorem, providing a mechanism for finding upper bounds on set sizes using partition rank, in the presence of a tensor that is constant on $k$-tuples with the same partition.

\begin{theorem}\label{thmmain1}
Let $f:\Pi_k \to \mathbb{F}$ and suppose $T:A^k \to \mathbb{F}$ is a tensor for which:
\begin{enumerate}
    \item[\emph{(i)}] $T(x_1,x_2,\ldots,x_k)=f(\pi)$ where $\pi$ is the partition of $(x_1,x_2,\ldots,x_k)$,
    \item[\emph{(ii)}] $f(\hat{1}) \neq -\sum_{\pi<\hat{1}} f(\pi) \cdot \mu(\pi,\hat{1})$.
\end{enumerate} 
Then 
\[
 \vert A \vert  \leq \mbox{partition-rank}(T)+(B_k-1)
\]
where $B_k$ is the $k$th Bell number, that is, the number of partitions of $\{1,2,\ldots,k\}$.
\end{theorem}
In many applications of the slice rank and partition rank methods, $k$ is fixed and the bounds on $ \vert A \vert $ of interest depend on data governing the size of $X$, which is typically independent of $k$. For instance, in \cite{crootlevpach},\cite{ge2020maximum},\cite{naslund2020partition},\cite{sauermann2021size}, $X=\Fqn$ and the asymptotics of interest are in $n$ and $q$. This makes the constant $B_k-1$ of neglible concern in most cases, even though it is large as a function of $k$. We remark however that if one genuinely has an interest in strengthening the summand $B_k-1$, it can be replaced by $\left \vert  \left\{\hat{0} < \tau < \hat{1} \ : \ \sum_{\pi \leq \tau} f(\pi) \cdot \mu(\pi,\tau) \neq 0 \right\} \right \vert $ since $\sum_{\pi \leq \tau} f(\pi) \cdot \mu(\pi,\tau)$ is the coefficient of $\delta_{\tau}(x_1,\ldots,x_k)$ in $I_f$.
\begin{proof}[Proof of Theorem~\ref{thmmain1}]
By Lemma~\ref{lempartitionindicator}, 
\[
I_f(x_1,\ldots,x_k)-T(x_1,\ldots,x_k) = \left(-\sum_{\pi<\hat{1}} f(\pi) \cdot \mu(\pi,\hat{1})-f(\hat{1}) \right) \cdot \delta_{\hat{1}}(x_1,\ldots,x_k).
\]
The tensor on the right is a diagonal tensor with all diagonal entries equal to $-\sum_{\pi<\hat{1}} f(\pi) \cdot \mu(\pi,\hat{1})-f(\hat{1})$, a nonzero constant by the second condition in the theorem. Subsequently, by Theorem~\ref{thmpartitionrank},
\begin{align*}
 \vert A \vert  &= \mbox{partition-rank}\left(\left(-\sum_{\pi<\hat{1}} f(\pi)  \cdot \mu(\pi,\hat{1})-f(\hat{1})\right) \cdot \delta_{\hat{1}} \right) \\
&= \mbox{partition-rank}\left( I_f-T\right) \\ &\leq \mbox{partition-rank}(I_f)+\mbox{partition-rank}(T).
\end{align*}
Every summand in $I_f$ has partition rank at most $1$ because for $\pi<\hat{1}$ with $\pi=\pi_1 \cdots \pi_r$, $\delta_{\pi}(x_1,\ldots,x_k)$ is $\prod_{i=1}^r \delta_{\pi_i}(x_1,\ldots,x_k)$. There are $B_k-1$ total summands so this gives us $ \vert A \vert  \leq \mbox{partition-rank}(T)+(B_k-1).$
\end{proof}
Our second main result, Theorem~\ref{thmmain2}, provides a general framework that mimics the approach of multiplying the distinctness indicator by a given tensor to make it diagonal. Instead, it looks at the more general setting of multiplying by a partition indicator $I_f$.

\begin{theorem}\label{thmmain2}
Let $A$ be a finite set and $\mathcal{P}$ be a property on $k$-tuples of $A$. Let $T:A^k \to \mathbb{F}$ and $f:\Pi_k \to \mathbb{F}$ for some field $\mathbb{F}$ be such that 
\begin{enumerate}
    \item[\emph{(i)}] $(x,\ldots,x)$ satisfies $\mathcal{P}$ for any $x \in A$,
    \item[\emph{(ii)}] $T(x_1,\ldots,x_k) = 1$ if $(x_1,\ldots,x_k)$ satisfies $\mathcal{P}$, and $0$ otherwise,
    \item[\emph{(iii)}] if there exists $(x_1,\ldots,x_k)$ with partition $\pi \in \Pi_k \backslash \{\hat{1}\}$ satisfying $\mathcal{P}$ then $f(\pi)=0$,
    \item[\emph{(iv)}] $\sum_{\pi < \hat{1}} f(\pi) \cdot \mu(\pi,\hat{1}) \neq 0$.
\end{enumerate}
Then
\[
 \vert A \vert  \leq \mbox{partition-rank}(I_f \cdot T).
\]
\end{theorem}
\begin{proof}
Let $(x_1,\ldots,x_k) \in A^k$. If the $x_i$'s are not all equal, then $(x_1,\ldots,x_k)$ has partition $\rho \neq \hat{1}$. By Lemma~\ref{lempartitionindicator}, $(I_f \cdot T)(x_1,\ldots,x_k) = f(\rho) \cdot T(x_1,\ldots,x_k).$ By conditions (ii) and (iii) of the theorem, we deduce $(I_f \cdot T)(x_1,\ldots,x_k)=0$. Now suppose $\rho=\hat{1}$, that is $x_1=\cdots=x_k$. Then also by Lemma~\ref{lempartitionindicator}, $(I_f \cdot T)(x_1,\ldots,x_k) =\left( - \sum_{\pi < \hat{1}} f(\pi) \cdot \mu(\pi,\hat{1}) \right) \cdot T(x_1,\ldots,x_k)$. This quantity is non-zero by conditions (i), (ii) and (iv) of the theorem. So altogether, $I_f \cdot T$ is diagonal with non-zero diagonal entries. The result follows from Theorem~\ref{thmpartitionrank}.
\end{proof}




\section{Applications}\label{secapplications}
This section is dedicated to applications of Theorem~\ref{thmmain1} and Theorem~\ref{thmmain2}. Unless otherwise specified, $q$ is an odd prime power.

Theorem~\ref{thmmain2} coalesces known applications of the partition rank method. To start, notice for any $T:A^k \to \mathbb{F}$ where distinct $k$-tuples in $A$ avoid a property $\mathcal{P}$, if conditions (i) and (ii) in the theorem are satisfied, the function $f:\Pi_k \to \mathbb{F}$ given by $f(\hat{0})=1$ and $f(\pi)=0$ for $\pi \in \Pi_k \backslash \{\hat{1}\}$ is a viable $f$ for the theorem provided that $\mu(\hat{0},\hat{1})=(-1)^{k-1}(k-1)!$ and $\mbox{char}(\mathbb{F})$ have no common factors. We saw that for this $f$, the partition indicator $I_f$ is the distinctness indicator (up to a sign change). This is what was used in \cite{naslund2020partition}.

Another more piercing instance in the literature can be seen in \cite{naslund2020exponential}. 


\subsection{Revisiting Erd\H{o}s-Ginzburg-Ziv}
A main concern in \cite{naslund2020partition} is determining an upper bound on the size of a subset $A \subseteq \mathbb{F}_p^n$ (with $p$ an odd prime) that has no distinct $p$-tuple $(\vec{x}_1,\ldots,\vec{x}_p)$ for which $\vec{x}_1+\cdots+\vec{x}_p=\vec{0}$. The purpose of this was to improve the then best known bound on the Erd\H{o}s-Ginzburg-Ziv constant for the abelian groups $\mathbb{F}_p^n$, which the smallest $k$ for which any sequence of elements of $\mathbb{F}_p^n$ of length $k$ contains a zero-sum subsequence of length $\mbox{exp}(\mathbb{F}_p^n)$. We show how an upper bound on the size of a subset $A$ with properties described above can be achieved by applying our framework to partition rank.

The tensor created to examine this was $T:A^p \to \mathbb{F}_p$ given by
\[
T(\vec{x}_1,\ldots,\vec{x}_p) = \prod_{s=1}^n \left(1-(\vec{x}_1(s)+\cdots+\vec{x}_p(s))^{p-1} \right).
\]
Here, $\vec{x}_j(s)$ is the $s$th coordinate of $\vec{x}_j$ for any $j$. Let $\mathcal{P}$ be the property that a $p$-tuple $(\vec{x}_1,\ldots,\vec{x}_p)$ sums to $\vec{0}$. Notice that $T$ is non-zero if and only if the property $\mathcal{P}$ is satisfied, but $T$ is  not necessarily diagonal. We amend this by applying Theorem~\ref{thmmain2}, the results of which recover those developed in \cite{naslund2020exponential}. We see that condition (ii) is satisfied since $T(\vec{x},\ldots,\vec{x})=p\vec{x}=\vec{0}$ for any $\vec{x}$. We might think to choose $f:\Pi_p \to \mathbb{F}_p$ in Theorem~\ref{thmmain2} so that on $\Pi_p \backslash \{\hat{1}\}$, $f(\hat{0})=1$ and $f(\pi)=0$ otherwise. This would recover the distinct indicator, so $ \vert A \vert  \leq \mbox{partition-rank}(I_f \cdot T)$. However, we can exploit the structure of $T$ to reduce the number of summands that appear in $I_f \cdot T$. The observation made in \cite{naslund2020exponential} is that if $(\vec{x}_1,\ldots,\vec{x}_p)$ has partition $\pi$ with exactly $2$ blocks, or equivalently rank $p-2$, then it is impossible for $\vec{x}_1+\cdots+\vec{x}_p=\vec{0}$: indeed if $m$ of the $p$ vectors are $\vec{x}$ and the other $p-m$ are another vector $\vec{y}$ then  $m\vec{x}+(p-m)\vec{y}=\vec{0}$, implying $m(\vec{x}-\vec{y})=\vec{0}$, which is impossible if $\vec{x},\vec{y}$ are distinct. So in Theorem~\ref{thmmain2}, we are free to select any choice of numbers in $\mathbb{F}_p$ for $f(\pi)$ when $\pi$ has exactly $2$ blocks. If we elect this freedom, set $f(\hat{0})=1$, and set $f(\pi)=0$ for all other $\pi$ with $1 \leq \mbox{rk}(\pi) \leq p-3$ we get
\[
I_f(\vec{x}_1,\ldots,\vec{x}_p) = \sum_{\substack{\tau \in \Pi_p \\ \tau < \hat{1}}} \left( \mu(\hat{0},\tau) + \sum_{\mbox{rk}(\pi)=p-2} f(\pi) \cdot \mu(\pi,\tau) \right) \delta_{\tau}(\vec{x}_1,\ldots,\vec{x}_p) \\
\]
When $\mbox{rk}(\pi)=p-2$ the coefficient of $\delta_{\pi}(\vec{x}_1,\ldots,\vec{x}_p)$ in the sum is $\mu(\hat{0},\pi)+f(\pi)\mu(\pi,\pi)$, so we can make this zero by setting $f(\pi)=-\mu(\hat{0},\pi)$. This makes summands of $I_f$ involve products of at least $3$ tensors, hence likely diminishing  $\mbox{partition-rank}(I_f \cdot T)$ and  lowering an upper bound on $ \vert A \vert $. This is exactly the choice made in \cite{naslund2020exponential}, up to a potential sign change. 
However by condition (iv) of Theorem~\ref{thmmain2}, this choice of $f$ works only under the condition that $\sum_{\pi < \hat{1}} f(\pi) \cdot \mu(\pi,\hat{1}) \neq 0$. We check this, and note that not surprisingly, the calculations recover what appears in Lemmas 7 and 8 in \cite{naslund2020exponential}. We see that
\begin{align*}
&\sum_{\pi < \hat{1}} f(\pi) \cdot \mu(\pi,\hat{1})  \\
&=f(\hat{0}) \mu(\hat{0},\hat{1}) \ + \sum_{\mbox{rk}(\pi)=p-2} -\mu(\hat{0},\pi) \cdot \underbrace{\mu(\pi,\hat{1})}_{-1}  \\ 
&=(-1)^{p-1}(p-1)! + \sum_{\mbox{rk}(\pi)=p-2} \mu(\hat{0},\pi).
\end{align*}
The elements in $\Pi_p$ with rank $p-2$ are the elements $\pi$ with two blocks $\pi_1,\pi_2$ say with sizes $\lambda_1 \geq \lambda_2$. Then $\lambda_2$ has fixed size $j \in \{1,\ldots,\lfloor p/2 \rfloor\}$ and there are $\binom{p}{j}$ ways to choose its contents. Furthermore, since $\pi_1$ has $p-j$ elements and $\pi_2$ has $j$ elements, the interval $[\pi,\hat{1}]$ is isomorphic to $\Pi_{p-j} \times \Pi_j$ so $\mu(\pi,\hat{1})=(-1)^{j-1}(j-1)!(-1)^{p-j-1}(p-j-1)!$. Doing this over all choices of $j$ elements for $\pi_2$ we get that the desired sum is 
\begin{align*}
&= (-1)^{p-1}(p-1)! + \left( \sum_{j=1}^{\lfloor p/2 \rfloor} \binom{p}{j} \cdot (-1)^{j-1}(j-1)!(-1)^{p-j-1}(p-j-1)! \right)  \\
&= (-1)^{p-1}(p-1)! + \left( \sum_{j=1}^{\lfloor p/2 \rfloor} (-1)^p \frac{p!}{j(p-j)} \right)  \\
\end{align*}
In the sum, each summand is divisible by $p$ (since $p$ is prime and for each $j$ we have $j,p-j<p$). So, the required sum is non-zero precisely if $(-1)^{p+1}(p-1)!$ is. By Wilson's Theorem this quantity is $(-1)^{p+2}=-1 \neq 0$. This indeed recovers the work in \cite{naslund2020exponential} (up to a potential sign change).

The previous example suggests a general  strategy for applying Theorem~\ref{thmmain2}. Suppose we are in the situation where the assumptions in Theorem~\ref{thmmain2} hold. Let $\Pi \subseteq \Pi_k \backslash \{\hat{1}\}$ be a set of partitions $\pi$ for which every $k$-tuple $(x_1,\ldots,x_k) \in A^k$ with partition $\pi$ has $T(x_1,\ldots,x_k)=0$.  For instance in the previous example we argued that we can take $\Pi=\{\hat{0}\} \cup \{\pi \in \Pi_k \ : \ \mbox{rk}(\pi)=k-2\}$. In the general setting, we claim for any such $\Pi$, we can find a function $f:\Pi_k \to \mathbb{F}$ that, on $\Pi_k \backslash \{\hat{1}\}$, is zero on $\pi \notin \Pi$, so that the coefficient of $\delta_{\pi}(x_1,\ldots,x_k)$ in $I_f(x_1,\ldots,x_k)$ is $0$ for all non-minimal $\pi \in \Pi$. 

To see this, set $f(\pi)=0$ for $\pi \notin \Pi$, and for partitions in $\Pi$ we induct on rank. If $\pi$ is minimal in $\Pi$, set $f(\pi)=1$. If $\pi$ is not minimal, $f(\tau)$ has been inductively selected for $\tau<\pi$. The coefficient of $\delta_{\pi}(x_1,\ldots,x_k)$ in $I_f(x_1,\ldots,x_k)$ is 
\[
\sum_{\tau \leq \pi} f(\tau) \cdot \mu(\tau,\pi) = f(\pi) + \sum_{\tau < \pi} f(\tau) \cdot \mu(\tau,\pi)
\] 
and we can make this $0$ by setting $f(\pi)=-\sum_{\tau < \pi} f(\tau) \cdot \mu(\tau,\pi)$. We see in fact that $f$ is unique up to selection of its values on minimal elements of $\Pi$. 

This process potentially eliminates many summands from $I_f \cdot T$, and in particular, summands $\delta_{\pi}(x_1,\ldots,x_k)$ for $\pi$ of high rank. This in turn can significantly reduce the partition rank of $I_f \cdot T$. The caveat here is certifying that $\sum_{\pi<\hat{1}} f(\pi) \cdot \mu(\pi,\hat{1}) \neq 0$; this depends heavily on the choice of $\Pi$ and the characteristic of the field $\mathbb{F}$.

In many situations, a tensor will not only be constant on $k$-tuples $(x_1,\ldots,x_k)$ with the same partition, but on $k$-tuples $(x_1,\ldots,x_k)$ with partitions of the same rank. Conditions of Theorem~\ref{thmmain2} can be stated succinctly in such situations. We first need a definition.

\begin{definition}\label{defpartitiongenerator}
Let $k,r$ be nonnegative integers with $0 \leq r \leq k-1$. The \emph{rank $r$ partition generator on $\Pi_k$} is the unique function $f:\Pi_k \to \mathbb{F}$ such that for $\pi \neq \hat{1}$,
\[f(\pi)=\begin{cases}
0  \hspace{0.25in} &\mbox{ if } \mbox{rk}(\pi) < r \\
1  \hspace{0.25in} &\mbox{ if } \mbox{rk}(\pi) = r \\
-\sum_{\tau<\pi} f(\tau) \cdot \mu(\tau,\pi)  \hspace{0.25in} &\mbox{ otherwise }
\end{cases}
\]
\end{definition}
The beauty of a rank $r$ partition generator on $\Pi_k$ is that condition (iv) in Theorem~\ref{thmmain2} becomes a completely combinatorial condition to check.
\begin{theorem}\label{thmcombcondition}
Let $0 \leq r \leq k-1$ be integers and let $f:\Pi_k \to \mathbb{F}$ be the rank $r$ partition generator on $\Pi_k$. Then \[-\sum_{\pi<\hat{1}} f(\pi) \cdot \mu(\pi,\hat{1})=S(k,k-r),\] the number of partitions of $\{1,2,\ldots,k\}$ into $k-r$ blocks.
\end{theorem}
\begin{proof}
From the definition of $f(\pi)$, it can be proven inductively that
\[
f(\pi) = \sum_{\substack{\ell \geq 1 \\ \pi_1< \cdots < \pi_{\ell}=\pi \\ \mbox{rk}(\pi_1) = r}}(-1)^{\ell+1} \mu(\pi_1,\pi_2) \cdots \mu(\pi_{\ell-1},\pi_{\ell}),
\]
and from this
\[
-\sum_{\pi<\hat{1}} f(\pi) \cdot \mu(\pi,\hat{1}) = \sum_{\substack{\ell \geq 1 \\ \pi_1< \pi_2 < \cdots < \pi_{\ell}<\hat{1} \\ \mbox{rk}(\pi_1) = r}} (-1)^{\ell} \mu(\pi_1,\pi_2) \cdots \mu(\pi_{\ell-1},\pi_{\ell}) \cdot \mu(\pi_{\ell},\hat{1}).
\]
Letting $I \in \mathbb{F}^{\Pi_k \times \Pi_k}$ be the identity matrix, this is equal to
\[
\sum_{\mbox{rk}(\pi)=r} \left( \sum_{\ell \geq 1} (-1)^{\ell} \cdot (\mu-I)^{\ell} \right) (\pi,\hat{1}) = \sum_{\mbox{rk}(\pi)=r} \left( \sum_{\ell \geq 1} (I-\mu)^{\ell} \right) (\pi,\hat{1}).
\]
Now 
\begin{align*}
\sum_{\ell \geq 1} (I-\mu)^{\ell} &= (I-\mu) \cdot \left( I + (I-\mu)+(I-\mu)^2+\dots \right) \\
&= (I-\mu) \cdot \left( I - (I-\mu) \right)^{-1} \\
&= (I-\mu) \cdot \mu^{-1} \\
&=\mu^{-1}-I = \zeta-I.
\end{align*}
so the desired sum is
\[
\sum_{\mbox{rk}(\pi)=r} (\zeta-I)(\pi,\hat{1}).
\]
This is precisely the total number of of rank $r$ elements in $\Pi_k$, which are the elements of $\Pi_k$ with $k-r$ blocks. This is precisely $S(k,k-r)$.
\end{proof}


\subsection{Acute Angles}

We move on to an example concretely illustrating how Theorem~\ref{thmmain1} applies. Erd\H{o}s \cite{erdos} conjectured that in the Euclidean space $\mathbb{R}^n$, any set of greater than $2^n$ points must create an obtuse angle, and this was proven by Danzer and Grunbaum \cite{danzer1962zwei}. This problem can be restated: any set in $\mathbb{R}^n$ whose distinct triples only forms acute or right angles has size at most $2^k$. A finite field analogue of this result was central to the work of Hart and Iosevich~\cite{hart2008ubiquity}. For an odd prime power $q$ and positive integer $n$, define $N(n,q)$ to be the maximum size of a subset of $A \subseteq \mathbb{F}_q^n$ such that any distinct triple of points in $A$ forms an acute angle. A consequence of their main Theorem 1.1 implies $N(n,q) = O\left(q^{\frac{n+1}{2}}\right)$. Shparlinski \cite{shparlinski2010point} improved this result for $n=2$, showing that $N(n,q) \leq 2q^{4/3}$. Observe both these upper bounds are exponential in $n$ for fixed $q$. We improve these by showing $N(n,q)$ is bounded above by a polynomial in $n$ for a fixed prime power $q$.

\begin{theorem}\label{thmacute}
Let $q$ be an odd prime power, and $A \subseteq \Fqn$. Then
\[
N(n,q) \leq \binom{n+q+1}{q-1}+4.
\]
\end{theorem}

Recall that a triple $\vec{x},\vec{y},\vec{z} \in \Fqn$ form an acute angle at $\vec{x}$ if and only if $2\langle \vec{x}-\vec{y}, \vec{x}-\vec{z} \rangle \in Q$ where $Q$ is the set of non-zero quadratic residues in $\Fq$ (see, for instance, \cite{shparlinski2010point}). The factor of $2$ here seems arbitrary, but as detailed in \cite{shparlinski2010point}, it is motivated by the observation in $\mathbb{R}^n$ that $\vec{x},\vec{y},\vec{z} \in \mathbb{R}^n$ form acute angle at $\vec{x}$ if and only if $\lvert \lvert \vec{x}-\vec{y} \rvert \rvert^2+\lvert \lvert \vec{x}-\vec{z} \rvert \rvert^2-\lvert \lvert \vec{y}-\vec{z} \rvert \rvert^2>0$. Then, identifying positive elements of $\mathbb{F}_q$ with non-zero quadratic residues, and observing 
\[
\lvert \lvert \vec{x}-\vec{y} \rvert \rvert^2+\lvert \lvert \vec{x}-\vec{z} \rvert \rvert^2-\lvert \lvert \vec{y}-\vec{z} \rvert \rvert^2 = 2\langle \vec{x}-\vec{y}, \vec{x}-\vec{z}\rangle,
\]
we see the condition for vectors in $\mathbb{F}_q^n$ to form acute angle naturally arises.

\begin{proof}[Proof (of Theorem~\ref{thmacute})]
Suppose $A \subseteq \Fqn$ and suppose that every triple of distinct points in $A$ forms an acute angled triangle. Let $T:A^3 \to \Fq$ be defined by 
\[
T(\vec{x},\vec{y},\vec{z})= \delta_{\vec{y} \neq \vec{z}} \ \cdot \  \prod_{\alpha \in Q} \left( 2\langle \vec{x}-\vec{y}, \vec{x}-\vec{z} \rangle - \alpha\right)^2.
\] 
First observe that if $\vec{x},\vec{y},\vec{z}$ are distinct then $\prod_{\alpha \in Q} \left( 2\langle \vec{x}-\vec{y}, \vec{x}-\vec{z} \rangle - \alpha \right)=0$ if and only if the vectors $\vec{x},\vec{y},\vec{z}$ form an acute angle at $\vec{x}$. So on $A^3$, $\prod_{\alpha \in Q} \left( 2\langle \vec{x}-\vec{y}, \vec{x}-\vec{z} \rangle - \alpha \right) = 0$ and it follows that $T(\vec{x},\vec{y},\vec{z})=0$ for distinct $\vec{x},\vec{y},\vec{z}$.  It is evident if $\vec{y}=\vec{z}$ then $T(\vec{x},\vec{y},\vec{z})=0$. If $\vec{x}=\vec{y}$ (or similarly $\vec{x}=\vec{z}$) and $\vec{y} \neq \vec{z}$ then 
\[
\delta_{\vec{y} \neq \vec{z}} \ \cdot \  \prod_{\alpha \in Q} \left( 2\langle \vec{x}-\vec{y}, \vec{x}-\vec{z} \rangle - \alpha\right)^2 =
 \prod_{\alpha \in Q} \left(  - \alpha\right)^2.
\]
Since $\prod_{\alpha \in Q} (-\alpha) \in \{\pm 1\}$, we deduce $T(\vec{x},\vec{y},\vec{z})=1$. Therefore,
\[
T(\vec{x},\vec{y},\vec{z}) = 
\begin{cases}
1  \hspace{0.25in} &\mbox{if } \vec{x}=\vec{y} \neq \vec{z} \mbox{ or } \vec{x}=\vec{z} \neq \vec{y} \\
0  \hspace{0.25in} &\mbox{otherwise}
\end{cases}
\]
We now apply Theorem~\ref{thmmain1}, recognizing that as partitions $\hat{0}=1 \vert 2 \vert 3$ and $\hat{1}=123$. Let $f:\Pi_3 \to \mathbb{F}_q$ be defined by 
\[
f(1 \vert 2 \vert 3)=0, \  f(1 \vert 23)=0,  \ f(2 \vert 13)=1, \  f(3 \vert 12)=1, \  f(123)=0.
\]
Then condition (i) of Theorem~\ref{thmmain1} holds. For condition (ii), we see that 
\[
-\sum_{\pi<\hat{1}} f(\pi) \cdot \mu(\pi,\hat{1})=-\mu(3 \vert 12,123)-\mu(2 \vert 13,123)=2
\]
which does not equal $f(\hat{1})=f(123)=0$ in $\Fq$ since $q$ is odd. It follows that $ \vert A \vert  \leq \mbox{partition-rank}(T)+(B_3-1) = \mbox{slice-rank}(T)+4$, the latter equality holding because slice rank and partition rank are equal on $3$ variables. First observe
\begin{align*}
&\prod_{\alpha \in Q } \left( 2\langle \vec{x}-\vec{y}, \vec{x}-\vec{z} \rangle - \alpha\right)^{2} =  \prod_{\alpha \in Q } \left(2 \cdot \sum_{i=1}^n x_i^2 - 2 \sum_{i=1}^n x_i \left(y_i+z_i\right) + 2\sum_{i=1}^n y_iz_i - \alpha \right)^{2} \\
&= \sum_{\beta+\beta'+\sum_{i=1}^n \gamma_i \leq 2 \vert Q \vert } \left(2 \sum_i x_i^2 \right)^{\beta} \left(2\sum_i y_iz_i \right)^{\beta'} \prod_{i=1}^n x_i^{\gamma_i} (y_i+z_i)^{\gamma_i} \cdot C_{\beta,\beta',\gamma_1,\cdots,\gamma_n}
\end{align*}
for some constants $C_{\beta,\beta',\gamma_1,\cdots,\gamma_n}$. Multiplying by $\delta_{y \neq z}$ we get
\[
\sum \left(2 \sum_i x_i^2 \right)^{\beta}  \prod_{i=1}^n x_i^{\gamma_i} \cdot \left[ \delta_{y \neq z} \prod_{i=1}^n (y_i+z_i)^{\gamma_i} \cdot \left(2\sum_i y_iz_i \right)^{\beta'} \cdot C_{\beta,\beta',\gamma_1,\cdots,\gamma_n} \right]
\]
\normalsize
which is written as a sum of slices, the number of which is bounded above by the number of nonnegative integers tuples $(\beta,\beta',\gamma_1,\ldots,\gamma_n)$ such that $\beta+\beta'+\gamma_1+\cdots+\gamma_n \leq 2 \cdot \vert Q \vert$. Since $\vert Q \vert=\frac{q-1}{2}$, by classical counting arguments this quantity is $\binom{(n+3)+(q-1)-1}{q-1} = \binom{n+q+1}{q-1}$. We deduce then that $\vert A \vert \leq \binom{n+q+1}{q-1}+4$.
\end{proof}


\subsection{Polynomial Bounds}

Our next application looks at a unification of many examples. Though it inherently uses the distinctness indicator rather than the full force of the partition indicator, it shows how using partition theory is useful in constructing bounds on the partition rank of a tensor. 

In many applications of the slice rank and partition rank methods, the tensors that are constructed are products of functions of the form $1-g^{q-1}$ where the functions $g$ are polynomials. When these $g$ exhibit symmetric behavior, we can apply Theorem~\ref{thmmain2} to get a bound on sets avoiding a particular property through the partition rank of these tensors. First we introduce the class of functions we are concerned with.

\begin{definition}\label{defidentifier}
Let $\mathcal{P}$ be a property defined on $k$-tuples from $\Fqn$, and $1 \leq m 
\leq k$. A function $g:(\Fqn)^m \to \Fq$ is an \emph{identifier} of $\mathcal{P}$ if 

\begin{enumerate}
    \item[(i)] $\vec{x}_1,\ldots,\vec{x}_k \in \Fqn$ satisfy property $\mathcal{P}$ if and only if $g(\vec{x}_{i_1},\ldots,\vec{x}_{i_m}) = 0$ for all $1 \leq i_1<\cdots<i_m \leq k$
    \item[(ii)] $g(\vec{x},\ldots,\vec{x})=0$ for any $\vec{x} \in \Fqn$.
\end{enumerate}
\end{definition}
Definition~\ref{defidentifier} is intrinsically defined in terms of a function $g$ and omits the need to talk about partition rank, so its benefit is it can be used together with the coming Theorem~\ref{thmpolynomial} to prove upper bounds on sets avoiding properties without having to know about slice or partition rank at all. Secondly, the set up in Definition~\ref{defidentifier} is not a contrived one in the least. Several examples in the literature that employ slice rank in fact start with an identifier of a property $\mathcal{P}$ with $m=2$. For instance, Pach, et al. \cite{pach2020avoiding} argue that identifying if a $3$-tuple of vectors $\vec{x}_1,\vec{x}_2,\vec{x}_3$ form a right triangle in $\Fqn$ is equivalent to determining if $g(\vec{x}_{i_1},\vec{x}_{i_2})=0$ for $1 \leq i_1 < i_2 \leq 3$ where $g(\vec{s},\vec{t})=\langle \  \vec{s}-\vec{t}, \vec{s}-\vec{t} \  \rangle$. 

In the case that a property $\mathcal{P}$ has an identifier in the sense of Definition~\ref{defidentifier}, and $k$ is fixed, Theorem~\ref{thmpolynomial} gives an  upper bound, polynomial in $n$, on the sizes of subsets of $\Fqn$ that avoid distinct $k$-tuples with property $\mathcal{P}$.

\begin{theorem}\label{thmpolynomial}
Fix a positive integer $k$ and let $q=p^{\ell}$ where $p \geq k$ is prime. Let $\mathcal{P}$ be a property defined on $k$-tuples of elements from $\Fqn$, and $g:(\Fqn)^m \to \Fq$ be an \emph{identifier} of $\mathcal{P}$ that is a polynomial with total degree $\deg(g)$. If $A \subseteq \Fqn$ and distinct $k$-tuples of $A$ avoid property $\mathcal{P}$ then 
\[
\vert A \vert =O\left(n^{k^{m-1} \deg(g) (q-1)}\right).
\]
\end{theorem}
\begin{proof}[Proof of Theorem~\ref{thmpolynomial}]
Let $T:A^k \to \Fq$ be given by
\[
T(\smallveclist) = \prod_{1 \leq i_1 < \cdots < i_m \leq k} \left( 1 - \left(g(\vec{x}_{i_1},\ldots,\vec{x}_{i_m}) \right)^{q-1} \right).
\]
Furthermore define $f:\Pi_k \to \mathbb{F}_q$ so that if $\pi<\hat{1}$ we have $f(\hat{0})=1$ and $f(\pi)=0$ otherwise. We check that $T,f$ satisfy the conditions of Theorem~\ref{thmmain2}. Condition (ii) in $g$ being an identifier of $\mathcal{P}$ gives that $g$ evaluated at the $m$-tuple $(\vec{x},\ldots,\vec{x})$ is $0$ for any $\vec{x} \in A$. By condition (i), this implies the $k$-tuple $(\vec{x},\ldots,\vec{x})$ has property $\mathcal{P}$ for $\vec{x} \in A$, so condition (i) of Theorem~\ref{thmmain2} is satisfied. For any $(\smallveclist) \in A^k$, $(\smallveclist)$ satisfies $\mathcal{P}$ if and only if $g(\vec{x}_{i_1},\ldots,\vec{x}_{i_m})=0$ for all $1 \leq i_1 < \cdots < i_m \leq k$ which happens if and only if $T(\smallveclist)=1$, so condition (ii) of Theorem~\ref{thmmain2} is satisfied. The only $\pi \in \Pi_k \backslash \{\hat{1}\}$ for which $f(\pi) \neq 0$ is $\pi = \hat{0}$ and $A$ evades distinct $k$-tuples with property $\mathcal{P}$ so condition (iii) of Theorem~\ref{thmmain2} is satisfied. Finally, $\sum_{\pi<\hat{1}} f(\pi) \cdot \mu(\pi,\hat{1}) = 1 \cdot \mu(\hat{0},\hat{1})=(-1)^{k-1}(k-1)!$ which is nonnzero in $\Fq$ because $p \geq k$.

It follows then by Theorem~\ref{thmmain2} that $ \vert A \vert  \leq \mbox{partition-rank}(I_f \cdot T)$ (and notice $I_f$ is actually the distinctness indicator). Expanding $I_f \cdot T$ we see that
\[
(I_f \cdot T) (\smallveclist)= \sum_{\substack{\tau \in \Pi_k \\ \tau < \hat{1}}} \left( \sum_{\pi \in \Pi_k} f(\pi)  \cdot \mu(\pi,\tau) \right) \cdot \delta_{\tau}(\vec{x}_1,\ldots,\vec{x}_k) \cdot   T(\smallveclist).
\]
so $\mbox{partition-rank}(I_f \cdot T)$ is bounded above by the sum of the partition ranks of all tensors $\delta_{\tau}(\smallveclist) \cdot T(\smallveclist)$ over all partitions of $\tau \in \Pi_k$ not equal to $\hat{1}$. Pick such a summand, using $\pi$ instead of $\tau$ as its label for consistency, and say $\pi=\pi_1\pi_2 \cdots \pi_r$. Then $r \geq 2$ since $\pi \neq \hat{1}$ and $\delta_{\pi}(\smallveclist) \cdot T(\smallveclist)$ is $\delta_{\pi}(\smallveclist) \cdot \prod_{1 \leq i_1 < \cdots < i_m \leq k} h(\vec{x}_{i_1},\ldots,\vec{x}_{i_m})$ where here, $h=1-g^{q-1}$ for simplicity. We first compute the degree of this summand.

Each polynomial factor that appears has degree at most $\deg(g) \cdot (q-1)$.  An upper bound on the desired degree then is $\deg(g) \cdot (q-1)$ mutiplied by the number of total polynomials appearing in the product

\begin{equation}\label{eqn:h}
\prod_{1 \leq i_1 < \cdots < i_m \leq k} h(\vec{x}_{i_1},\ldots,\vec{x}_{i_m})
\end{equation} 

To bound the latter, we observe that we can identify variables $\vec{x}_i$ and $\vec{x}_j$ when $i$ and $j$ are in the same part of $\pi$, for otherwise $\delta_{\pi}(\smallveclist)=0$. Once variables are identified, any occurrence of $h^N$ in (\ref{eqn:h}) with $N \geq 2$ may be replaced by $h$ since $h$ takes on values $0$ or $1$. So the number of distinct polynomials that can arise in the product after this process is bounded above by the number of ways to assign $m$ ordered variables to the $r$ parts of $\pi$ (and then identifying variables that are in the same part). There are $r^m$ many such assignments. Subsequently the degree of the polynomial in $(\ref{eqn:h})$ is bounded above by $r^m \deg(g)(q-1)$. 


Let $j_t$ be the smallest element of $\pi_t$ for each $1 \leq t \leq r$. Because of $\delta_{\pi}$, we can write the product as a polynomial in the coordinates of $\vec{x}_{j_1},\cdots,\vec{x}_{j_r}$. After expanding, each term will be of the form
\scriptsize
\begin{equation*}
\left[ \delta(\vec{x}_{\pi_1}) \vec{x}_{j_1}(1)^{e_{11}} \cdots \vec{x}_{j_1}(n)^{e_{1n}} \right]  \cdots \left[ \delta(\vec{x}_{\pi_r})  \vec{x}_{j_r}(1)^{e_{r1}} \cdots \vec{x}_{j_r}(n)^{e_{rn}} \right],
\end{equation*}
\normalsize
so one of the components of one of the vectors $\vec{x}_{j_1},\ldots,\vec{x}_{j_r}$ contributes at most $\frac{1}{r} r^m  \deg(g) \cdot (q-1) = r^{m-1} \cdot \deg(g) \cdot (q-1) $ to the degree of such a term. It follows that the partition rank of $\delta_{\pi}(\smallveclist) \cdot T(\smallveclist)$ is bounded above by $r$ times the number of monomials in $n$ variables with degree at most $r^{m-1}  \deg(g)  (q-1)$ which is $r \cdot \binom{n+r^{m-1}  \deg(g)  (q-1)}{r^{m-1}  \deg(g)  (q-1)}$. Doing this over all $\pi \in \Pi_k \backslash \{\hat{1}\}$ we get
\begin{align*}
\mbox{partition-rank}(I_f \cdot T) &\leq \sum_{r=2}^k S(k,r) \cdot r \cdot \binom{n+r^{m-1} \deg(g) (q-1)}{r^{m-1}  \deg(g)  (q-1)} \\
&\leq \left( \sum_{r=2}^k r \cdot S(k,r) \right) \binom{n+k^{m-1} \deg(g) (q-1)}{k^{m-1}  \deg(g)  (q-1)},
\end{align*}
the last inequality holding because $\binom{n+s}{s}$ increases as $s$ increases. It follows that $\vert A \vert=O(n^{\gamma})$ where $\gamma = k^{m-1}\deg(g)(q-1).$

\end{proof}

\vspace{0.2in} 

We now apply Theorem~\ref{thmpolynomial} to geometric phenomena in $\Fqn$. The first of these is a generalization of the work of Pach, et al. \cite{pach2020avoiding} on avoiding right triangles. Recall their result:

\begin{theorem}[\cite{pach2020avoiding}]\label{thmrighttriangle}
If $A \subseteq \Fqn$ contains no triangle with all right angles, i.e. vectors $\vec{x},\vec{y},\vec{z}$ with
\[
\langle \vec{x}-\vec{y},\vec{y}-\vec{z} \rangle = \langle \vec{y}-\vec{z},\vec{z}-\vec{x} \rangle = \langle \vec{z}-\vec{x},\vec{x}-\vec{y} \rangle = 0
\]
then $ \vert A \vert =O\left(n^{2(q-1)}\right)$.
\end{theorem}

We can use Theorem~\ref{thmpolynomial} to give a more general result that subsets avoiding right-angled $k$-configurations also have size polynomial in $n$.
\begin{theorem}\label{thmrightsimplex}
Fix a positive integer $k$ and let $q=p^{\ell}$ where $p \geq k$ is prime. If $A \subseteq \Fqn$ contains no $k$-configuration with all right angles, that is, distinct vectors $\smallveclist$ such that $\langle \vec{x}_i-\vec{x}_j,\vec{x}_j-\vec{x}_{\ell} \rangle = 0$ for distinct triples $i,j,\ell$, then $ \vert A \vert =O\left(n^{2k(q-1)}\right)$.
\end{theorem}
\begin{proof}
From \cite{pach2020avoiding}, the property $\mathcal{P}$ of $\vec{x}_1,\ldots,\vec{x}_k,\vec{x}_{k+1}$ having $\langle \vec{x}_i-\vec{x}_j,\vec{x}_j-\vec{x}_{\ell} \rangle = 0$ for all triples $i,j,\ell$ is equivalent to demanding that for every distinct $i,j$, $\langle \vec{x}_i-\vec{x}_j,\vec{x}_i-\vec{x}_j \rangle = 0$. Define $g:(\Fqn)^2 \to \Fq$ by $g(\vec{s},\vec{t})=\langle \vec{s}-\vec{t},\vec{s}-\vec{t} \rangle$. We see that $g(\vec{x}_i,\vec{x}_j)=0$ for all $i<j$ if and only if $\smallveclist$ has property $\mathcal{P}$. Furthermore, we observe that if all $\veclist$ are identical then $g(\vec{x}_i,\vec{x}_j)=g(\vec{x}_i,\vec{x}_i)=0$ for all $i \neq j$. This justifies that $g$ is an identifier of $\mathcal{P}$ in the sense of Definition~\ref{defidentifier}. So if $A$ avoids property $\mathcal{P}$ then Theorem~\ref{thmpolynomial} implies $ \vert A \vert =O(n^{\gamma})$ where $\gamma=k^{2-1} \cdot 2(q-1) = 2k(q-1)$.
\end{proof}

We begin to see through this application that the generality of Theorem~\ref{thmpolynomial} enforces a trade-off between applicability to many scenarios, and optimality based on specificity in each scenario. For instance, several techniques can be used to possibly optimize the exponent in Theorem~\ref{thmrightsimplex}. The partition rank upper bound in Theorem~\ref{thmpolynomial} has the contribution of $\delta_{\pi_i}$ for a particular block $\pi_i$ of a partition appearing frequently. Techniques such as those in \cite{naslund2022chromatic} can be used to possibly reduce the exponent. Despite this, Theorem~\ref{thmpolynomial} still acts as an automatic certification that sets of size polynomial in $n$ are forced have certain properties.

We can extend these ideas to proving that subsets of $\Fqn$ avoiding $k$-tuples whose  squared distances from each other are equal mod $p$ (where here $q=p^{\ell}$ for some natural number $\ell$) must have size polynomial in $n$ as well.

\begin{theorem}
Fix a positive integer $k$ and let $q=p^{\ell}$ where $p \geq k$ is prime. Suppose $A \subseteq \Fqn$ is a set that does not contain a distinct $k$-tuple whose squared distances from each other are all equal mod $p$. Then $ \vert A \vert =O(n^{\gamma})$ where $\gamma=2k^2(q-1)$.
\end{theorem}
\begin{proof}
Let $g:(\F_q)^3 \to \Fq$ be defined by $g(\vec{x}_i,\vec{x}_j,\vec{x}_{\ell})= \vert  \vert \vec{x}_i-\vec{x}_j \vert  \vert ^2- \vert  \vert \vec{x}_j-\vec{x}_{\ell} \vert  \vert ^2$. Then $g$ is an identifier of the property $\mathcal{P}$ of a set of $k$ points having squared distances between pairs of them all equal mod $p$. Indeed if $\smallveclist$ satisfy $\mathcal{P}$ then $g(\vec{x}_{i_1},\vec{x}_{i_2},\vec{x}_{i_3})=0$ for all $1 \leq i_1 < i_2 < i_3 \leq k$ since $ \vert  \vert \vec{x}_{i_1}-\vec{x}_{i_2} \vert  \vert ^2- \vert  \vert \vec{x}_{i_2}-\vec{x}_{i_3} \vert  \vert ^2$. Furthermore,  if $\vec{x}_{i_1}=\vec{x}_{i_2}=\vec{x}_{i_3}$ then $g(\vec{x}_{i_1},\vec{x}_{i_2},\vec{x}_{i_3})=0$ for all $1 \leq i_1 < i_2 < i_3 \leq k$. Then by Theorem~\ref{thmpolynomial}, if $A \subseteq \Fqn$ and any $k$-tuple in $A$ avoids property $\mathcal{P}$,  $ \vert A \vert =O\left(n^{2k^2(q-1)}\right)$.
\end{proof}

Our last geometric application establishes a polynomial bound on the size of subsets of $\Fqn$ consisting of self-orthogonal points, while avoiding distinct $k$-tuples of pairwise orthogonal points. This complements work of Iosevich and Senger \cite{senger} who investigate a related enumerative question, showing that if subsets of $\Fqn$ are large enough than they contain any $k$-tuples of pairwise orthogonal points.

\begin{theorem}
Fix a positive integer $k$ and let $q=p^{\ell}$ where $p \geq k$ is prime. Let $A \subseteq \Fqn$ be a set of self-orthogonal points. Suppose $A$ does not contain a distinct $k$-tuple of pairwise orthogonal points. Then $ \vert A \vert =O\left(n^{2k(q-1)}\right)$.
\end{theorem}

\begin{proof}
Define $g:(\Fqn)^2 \to \Fq$ by $g(\vec{x},\vec{y})=\vec{x} \cdot \vec{y}$. Then $g$ is an indicator for the property of $\mathcal{P}$ of having a $k$-tuple of pairwise orthogonal points. Furthermore, the second condition of $g$ being an identifier of $\mathcal{P}$ is satisfied since $A$ consists of self-orthogonal points.  By Theorem~\ref{thmpolynomial}, $ \vert A \vert =O\left(n^{2k(q-1)}\right)$.
\end{proof}

\section{Future Directions}\label{secfuturedirections}
Many pressing questions arise from this new perspective on partition rank. Of the most pertinent is how to apply our main theorems, Theorem~\ref{thmmain1} and Theorem~\ref{thmmain2} in more general settings than the ones presented. There are several obstacles to this achievement. In applying Theorem~\ref{thmmain1}, one needs to construct a tensor that is constant on $k$-tuples with the same partition. For many tensors that identify combinatorial properties, the value on a $k$-tuple depends heavily on the combinatorics of the $k$-tuple beyond just which coordinates are equal. For instance if the coordinates are themselves are vectors then the value of the tensor might depend on the matroid of these vectors. However, we still believe there are a wide variety of settings in which Theorem~\ref{thmmain1} applies. 

\begin{problem}
Discover new bounds on sets avoiding combinatorial properties by constructing tensors that are constant on $k$-tuples of points with the same partition.
\end{problem}

It would be particularly interesting to see examples in which $k \geq 4$, since for $k \leq 3$ the slice rank and partition rank agree. Finally, it is a curiosity how widespread the applications of Theorem~\ref{thmpolynomial} can be. Are there other geometric phenomena $\mathcal{P}$ in $\mathbb{F}_q^n$ that can be naturally posed through an identifier $g$ of $\mathcal{P}$. We suspect there are many possibilities here.

\section{Acknowledgments}

The author thanks the anonymous referees for feedback on the manuscript, particularly improving the presentation and details of Section~\ref{secapplications} of the article. The author thanks Bruce Sagan for fruitful conversation in preparing the manuscript. This work was partially supported by the AMS Claytor-Gilmer Fellowship and the Karen EDGE Fellowship.


\end{document}